\documentclass[oneside]{amsart}
\usepackage{amsmath, amssymb}
\usepackage{amsthm}
\newtheorem{theorem}{Theorem}[section]

\newtheorem{lemma}[theorem]{Lemma}
\newtheorem{corollary}[theorem]{Corollary}

\theoremstyle{definition}

\def\R{\mathbb{R} } 
\def\Z{\mathbb{Z} } 
\def\nbd{neighborhood } 
 
\def\R{\mathbb{R} }

\begin{document}

\title[$R$-closed homeomorphisms on surfaces]
{$R$-closed homeomorphisms on surfaces}
\author{Tomoo Yokoyama}
\date{\today}
\email{yokoyama@math.sci.hokudai.ac.jp}

\thanks{The author is partially supported 
by the JST CREST Program at Department of Mathematics,  
Hokkaido University.}

\maketitle

\begin{abstract}
Let $f$ be an $R$-closed homeomorphism on 
a connected orientable closed surface $M$. 
In this paper, 
we show that 
if $M$ has genus more than one,  
then each minimal set 
is either a periodic orbit 
or an extension of a Cantor set. 
If $M = \mathbb{T}^2$ and 
$f$ is neither minimal nor periodic, 
then either 
each minimal set is 
a finite disjoint union of essential circloids 
or  
there is a minimal set which is  
an extension of a Cantor set. 
If $M = \mathbb{S}^2$ and $f$ is not periodic 
but orientation-preserving (resp. reversing), 
then the minimal sets of $f$ (resp. $f^2$) are 
exactly two fixed points and 
other circloids  
and $\mathbb{S}^2/\widetilde{f} \cong [0, 1]$. 
\end{abstract}

\section{Introduction}
In \cite{Ma}, 
it has shown that 
if
$f$ is orientation-preserving $R$-closed and non-periodic 
homeomorphism on $\mathbb{S}^2$, 
then $f$ has exactly two fixed points and every non-degenerate orbit
closure is a homology $1$-sphere. 
In this paper, 
we consider minimal sets of $R$-closed 
homeomorphisms on closed surfaces. 
Precisely, 
let $f$ be an $R$-closed homeomorphism on 
a connected orientable closed surface $M$. 
Then we show that 
if $M$ has genus more than one,  
then each minimal set 
is either a periodic orbit 
or an extension of a Cantor set. 
If $M = \mathbb{T}^2$ and 
$f$ is neither minimal nor periodic, 
then either 
the orbit class space $\mathbb{T}^2/\widetilde{f}$ 
is a $1$-manifold 
and 
each minimal set is 
a finite disjoint union of essential circloids, 
or  
there is a minimal set which is  
an extension of a Cantor set. 
If $M = \mathbb{S}^2$ and $f$ is not periodic 
but orientation-preserving (resp. reversing), 
then the minimal sets of $f$ (resp. $f^2$) are 
exactly two fixed points and 
other circloids  
and $\mathbb{S}^2/\widetilde{f} \cong [0, 1]$. 
Finally we state the applications for codimension two foliations.

\section{Preliminaries}

By a flow, we mean a continuous action of a topological group $G$ 
on a topological space $X$. 
We call that 
$G$ is $R$-closed if $R := \{ (x, y) \mid y \in \overline{G(x)} \}$ is closed. 
%
%
Recall that a subset $S$ of $G$
is is said to be (left) syndetic if there is a compact set $K$ of $G$ with $KS = G$. 
For a point $x \in  X$ 
and an open $U$ of $X$,  
let $N(x, U) = \{ g \in G \mid  gx \in  U\}$.  
We say that 
$x$ is an almost periodic point if 
$N(x, U)$ is syndetic for every neighborhood $U$ of $x$. 
A flow $G$ is pointwise almost periodic if 
every point $x \in X$ is almost periodic. 
When $X$ is  a compact metrizable (i.e. compact Hausdorff) space, 
they are known that  
if $f$ is $R$-closed, 
then 
$f$ is pointwise almost periodic,  
and that  
$f$ is pointwise almost periodic 
if and only if 
$\{ \overline{G(x)} \mid x \in X \}$ is a decomposition of $X$. 
%
%
%
When $f$ is pointwise almost periodic, 
write 
$\hat{\mathcal{F}} := \{  \overline{G(x)} \mid x \in X \}$ 
the decomposition of $X$. 
Note that 
$X/\hat{\mathcal{F}}$ 
is called an orbit class space 
and is also denoted by $X/ \widetilde{G}$.  
%
A pointwise almost periodic flow $G$ is 
weakly almost periodic in the sense of Gottschalk \cite{G} 
if 
the saturation of orbit closures for any closed subset of $X$ is 
closed (i.e. the quotient map $X \to X/\widetilde{G}$ is closed). 
By Theorem 5 \cite{Ma} and Proposition 1.2 \cite{Y}, 
the following are equivalent for 
a pointwise almost periodic flow $f$ on a compact metrizable space: 
1) $\hat{\mathcal{F}}$ is $R$-closed, 
2) $\hat{\mathcal{F}}$ is weakly almost periodic, 
3) $\hat{\mathcal{F}}$ is upper semi-continuous,  
4) $X/\hat{\mathcal{F}}$ is Hausdorff.

By a continuum we mean a compact connected metrizable space which 
is not a singleton. 
A continuum $A \subset X$ is said to be annular 
if it has a neighbourhood $U \subset X$
homeomorphic to an open annulus 
such that $U - A$ has exactly two components each 
of which is homeomorphic
to an annulus. 
We call any such $U$ an annular neighbourhood of $A$.
We say a subset $C \subset X$ is a circloid 
if it is an annular continuum 
and 
does not contain any strictly smaller annular continuum as a subset.
For 
a subset $A$ of $X$ and 
a decomposition $\hat{\mathcal{F}}$, 
the saturation $\mathrm{Sat}(A)$ of $A$ is 
the union 
$\cup \{ L \in \hat{\mathcal{F}} \mid A \cap L \neq \emptyset \}$ of 
elements of $\hat{\mathcal{F}}$ intersecting $A$.

\begin{lemma}\label{lem:01}
Let $X$ be a sequentially compact space and 
$(C_n)$ a sequence of connected subsets of $X$. 
Suppose that 
there are disjoint open subsets $U, V$ of $X$ and 
sequences $(x_n)$ (resp. $(y_n)$) converging to $x \in U$ (resp. $y \in V$) 
with $x_n, y_n \in C_n$. 
Then there is an element $z \in ( \cap_{n>0} \overline{\cup_{k >n}C_k}) \setminus U \sqcup V$. 
\end{lemma}

\begin{proof} 
Let $F = X - U \sqcup V$ be a closed subset. 
Since $C_n$ is connected, 
each $C_n$ intersects $F$. 
Choose $z_n \in C_n \cap F$. 
Since $X$ is sequentially compact, 
we have 
$F$ is also sequentially compact. 
Hence there is a convergent subsequence of $z_n$ and 
so the limit $z \in F$ is desired.  
\end{proof} 

We show that 
connected closures  
for an $R$-closed flow 
must converge to a connected closure. 

\begin{lemma}\label{lem:02}
Let $G$ be an $R$-closed flow  on 
a sequentially compact space $X$ and 
let $(w_n)$ be 
a convergent sequence to a point $w \in M$. 
If 
each $\overline{G(w_n)}$ is connected, 
then the closure $\overline{G(w)}$ 
is connected. 
\end{lemma}

\begin{proof} 
Put $C_n := \overline{G(w_n)}$. 
%
Suppose that  
$\overline{G(w)}$ is disconnected. 
Then there are disjoint open subsets $U, W$ of $M$ 
such that 
$\overline{G(w)} \subseteq U \sqcup W$, 
$\overline{G(w)} \cap U \neq \emptyset$, 
and 
$\overline{G(w)} \cap W \neq \emptyset$. 
Then 
$G(w) \cap U \neq \emptyset$ 
and 
$G(w) \cap W \neq \emptyset$. 
Since $w_n$ converges to $w$, 
the continuity of $G$ implies that 
there are sequences $(x_n)$ (resp. $(y_n)$) 
converging $x \in U$ (resp. $y \in W$) 
such that $x_n, y_n \in C_n$.
%
By Lemma \ref{lem:01}, 
there is an element 
$z \in ( \cap_{n>0} \overline{\cup_{k >n}C_k}) \setminus U \sqcup W$.  
Then 
there is a convergent sequence 
$(z_n \in C_n)$ to $z$. 
Since 
$z_n \in C_n = \overline{G(w_n)}$
and 
$z \notin \overline{G(w)}$,  
we obtain 
$( w_n, z_n) \in R$ and 
$(w, z) \notin R$.  
This contradicts the $R$-closedness. 
Therefore 
$C$ is connected. 
\end{proof}

Let $f$ be a pointwise almost periodic homeomorphism on 
an orientable connected closed surface $M$. 
Recall $\hat{\mathcal{F}}  = \hat{\mathcal{F}}_f 
= \{ \overline{O_f(x)} \mid x \in M \}$. 
Write 
$V = V_f := \{ x \in M - \mathrm{Fix} (f) \mid \overline{O(x)} \text{ is connected } \} 
= \cup \{ L \in \hat{\mathcal{F}} : \text{ connected } \} - \mathrm{Fix} (f)$. 

\begin{lemma}\label{lem:0125}
If $f$ is not minimal, 
then $V$ consists of circloids.
\end{lemma}

\begin{proof} 
Since  $f$ is pointwise almost periodic, we have that 
the non-wandering set of $f$ is $M$.  
By Theorem 1.1.\cite{K}, each element $C$ in $V$ of $\hat{\mathcal{F}}$ is 
annular.  
Let $U$ be a sufficiently small annular neighbourhood of $C$ such that 
$U - C$ is a disjoint union of two open annuli $A_1, A_2$. 
Since $C$ is $f$-invariant and minimal, 
we have that 
$C = \partial A_1 \cap \partial A_2$. 
Suppose that 
there is an annular continuum $C' \subsetneq C$. 
Then there is an annular neighbourhood $U'$ of $C'$ such that 
$U' \subset U$. 
Embedding $U$ into $\mathbb{S}^2$, 
we may assume that $U$ is a subset of $\mathbb{S}^2$. 
Then 
$\mathbb{S}^2 - C$ is a disjoint union of two open disks $D_1, D_2$
and   
$\mathbb{S}^2 - C'$ is a disjoint union of two open disks $D'_1, D'_2$. 
Since $\mathbb{S}^2 - C' \supsetneq \mathbb{S}^2 - C$, 
we have 
$D_1 \sqcup D_2 \subsetneq
D'_1 \sqcup D'_2$. 
%
Since 
$ D_1 \sqcup  D_2 \sqcup \{ x \}$ for 
any element $x \in C$ is connected,  
we obtain 
$D'_1 \sqcup  D'_2$ is connected. 
This contradicts to disconnectivity. 
Thus  
$C$ is a circloid.   
\end{proof}

%
%
Note that 
a point $x$ is almost periodic if and only if 
for every open neighborhood $U$ of $x$, 
there is $K \in \mathbb{Z}_{\geq 0}$ such that 
$\mathbb{Z} = \{ n, n + 1, \dots, n + K \mid n \in N(x, U) \}$.  
The above lemmas implies the following statement.

\begin{corollary}\label{cor:024}
Suppose $f$ is not minimal but $R$-closed. 
Each point of $\overline{V} -V$ is a fixed point. 
\end{corollary}

Taking the iteration, 
we obtain the following corollary.

\begin{corollary}\label{cor:02}
Suppose $f$ is not minimal. 
For any $x \in M$, 
if $\overline{O(x)}$ is 
not periodic 
but  consists of finitely many connected components, 
then $\overline{O(x)}$ consists of circloids.
\end{corollary}

\begin{proof} 
Let $k$ be the number of the connected components of $\overline{O(x)}$. 
By Theorem I \cite{ES}, 
we have $f^k$ is also pointwise almost periodic and 
so $\overline{O_{f^k}(x)}$ is connected. 
By Lemma \ref{lem:0125}, 
$\overline{O_{f^k}(x)}$ is a circloid. 
Since each connected component of $\overline{O(x)}$ 
is homeomorphic to each other, 
the assertion holds. 
\end{proof}

This corollary can sharpen Theorem 6 \cite{Ma} 
into the following statement.

\begin{corollary}\label{lem:0155}
Let $f$ be a non-periodic $R$-closed orientation-preserving (resp. reversing) 
homeomorphism on $\mathbb{S}^2$. 
Then 
$\mathbb{S}^2/ \widetilde{f}$ is a closed interval  
 and 
$\hat{\mathcal{F}}_f$ (resp. $\hat{\mathcal{F}}_{f^2}$)  
consists of two fixed points and other circloids. 
\end{corollary}

\begin{proof} 
Suppose that $f$ is orientation-preserving. 
%
By Theorem 3 and 6 \cite{Ma}, 
there are exactly two fixed points and 
all other orbit closures of $f$ are connected. 
By Lemma \ref{lem:0125}, 
they are circloids. 
We show that 
$M/ \hat{\mathcal{F}}$ is a closed interval. 
Indeed, 
let $A$ be the sphere minus two fixed points. 
Suppose that 
there is a circloid $L$
which is null homotopic in $A$. 
Let $D$ be a disk bounded by $L$ in $A$. 
Since $M$ consists of non-wandering points, 
the Brouwer's non-wandering Theorem \cite{B} to $D$  
implies that 
$f|_{D}$ has a fixed point. 
This contradicts to the non-existence of fixed points in $A$. 
On the orientation reversing case, 
since $f^2$ is orientation-preserving, 
the assertion holds. 
\end{proof}

Note that 
if there is a dense orbit and 
$\hat{\mathcal{F}}$ is pointwise almost periodic, 
then 
$\hat{\mathcal{F}}$ is minimal and $V = T^2 = M$. 
Now we proof a key lemma.

\begin{lemma}\label{lem:01555}
Suppose that 
$f$ is an orientation-preserving (resp. reversing) 
$R$-closed homeomorphism on an orientable connected closed surface $M$. 
If  
there is a minimal set which is a circloid, 
then 
$M/\hat{\mathcal{F}} = \overline{V/\hat{\mathcal{F}}}$ 
is a closed interval or a circle. 
Moreover 
either 
$M \cong \mathbb{S}^2$ and 
$\hat{\mathcal{F}}_f$ 
(resp. $\hat{\mathcal{F}}_{f^2}$) consists of 
exactly two fixed points and other circloids, 
or 
$M \cong \mathbb{T}^2$ and 
$\hat{\mathcal{F}}_f$ 
(resp. $\hat{\mathcal{F}}_{f^2}$) consists of 
essential circloids. 
\end{lemma}

\begin{proof} 
Fix a metric compatible to the topology of $M$. 
First, suppose that 
$f$ is orientation-preserving. 
First we show that $V$ is open. 
Let $L$ be a 
circloid of $\hat{\mathcal{F}}$ 
with a sufficiently small annular neighbourhood $A$. 
Since $\hat{\mathcal{F}}$ is $R$-closed,  
Lemma 1.6 \cite{Y} implies that 
the quotient map $M \to M/\hat{\mathcal{F}}$ is closed 
and so 
the saturation $F : = \mathrm{Sat}(M - A)$ is closed and nonempty. 
Since $M/\hat{\mathcal{F}}$ is compact Hausdorff, 
there is a small number $\varepsilon > 0$ 
such that 
$F$ does not intersect with 
the closure of 
the $\varepsilon$-\nbd  $B_{\varepsilon}(L)$ 
of $L$. 
Since $\hat{\mathcal{F}}$ is upper semi-continuous, 
there is an open saturated neighbourhood $U \subseteq B_{\varepsilon}(L)$ 
of $L$. 
Since $F \supseteq M - A$ and 
$F \cap  \overline{B_{\varepsilon}(L)} = \emptyset$, 
we have 
$\overline{U} \subseteq A$. 
Since $M$ is compact metrizable 
and since  
$M - U$ and $L$ are disjoint closed, 
there is some $\varepsilon' \in (0, \varepsilon)$ such that 
the open $\varepsilon'$-\nbd $V := B_{\varepsilon'}(L)$ of $L$ 
contained in $U$. 
Then $V$ is arcwise connected. 
Since $L$ is invariant and contained in $V$, 
we have that 
$L \subset f^k(V)$ for any $k \in \Z$ and so  
the saturation $\mathrm{Sat}(V) = \cup_{k \in \Z} f^k(V)$
 is arcwise connected open  
and is contained in $U$. 
By replacing $U$ by $\mathrm{Sat}(V)$, 
we may assume that 
$U$ is arcwise connected. 
Define 
$\mathop{\mathrm{Fill}}(U) : = \cup \{ B  \subset A : \text{disk} \mid
 \partial B = \gamma \text{ for some loop } \gamma  \subset U\}$. 
We show that 
$\mathop{\mathrm{Fill}}(U)$ is annular.  
Indeed, 
since $U \subseteq B_{\varepsilon}(L)$, 
we have $\mathop{\mathrm{Fill}}(U) \subseteq B_{\varepsilon}(L) \subseteq A$. 
Since $A$ is an open annulus, 
we can consider an embedding of $A$ to a sphere $S$  
such that the complement of $L$ consists of 
two open disks $D_1, D_2$. 
Then $A \cup D_i$ is an open disk. 
For each connected component $B$ of $A - U$ which is contained by 
a disk bounded by a loop in $U$, 
the saturation $\mathrm{Sat}(B)$ is contained in a union of disks 
bounded by loops in $U$. 
Therefore 
$\mathop{\mathrm{Fill}}(U) \cup D_i 
 = \cup \{ B  \subset A \cup D_i : \text{disk} \mid
 \partial B = \gamma \text{ for some loop } \gamma  \subset U \cup D_i \}
$ 
is a simply connected open subset of the disk $A \cup D_i$. 
By Riemann mapping theorem, 
this implies that 
$\mathop{\mathrm{Fill}}(U) \cup D_i$ is homeomorphic to an open disk 
and so $\mathop{\mathrm{Fill}}(U)$ is annular. 
Note that $\mathop{\mathrm{Fill}}(U)$ is saturated. 
Let $N$ be the two points compactification of $\mathop{\mathrm{Fill}}(U)$  
and $f'$ the resulting homeomorphism on $N$ adding 
new two fixed points. 
Then $N$ is a sphere and 
$L$ is also a minimal set of $f'$ which separates the new two fixed points. 
Since $M/\widetilde{f}$ is normal, 
the closedness of $M - \mathop{\mathrm{Fill}}(U)$ 
implies that 
$N/\widetilde{f'}$ is Hausdorff. 
Hence 
we have $f'$ is $R$-closed. 
By Corollary \ref{lem:0155}, 
the orbit closures on $N$ are 
exactly two fixed points and other circloids. 
This implies that 
$V$ is open and 
the quotient of each connected component of $V$ is totally ordered. 
By Corollary \ref{cor:024}, 
each boundary of $V$ is a fixed point 
and so 
$M/ \mathcal{F}$ is an closed interval or a circle. 
If $V$ has nonempty boundaries, 
then 
$M$ is a sphere 
and 
if $V$ has no boundaries 
then $M$ is a torus. 
This completes a proof of the orientable case. 
Suppose that 
$f$ is orientation-reversing. 
Since $f^2$ is orientation-preserving 
and since 
$M/\hat{\mathcal{F}}_{f^2}$ 
is a double branched covering of 
$M/\hat{\mathcal{F}}$, 
we have that 
$M/\hat{\mathcal{F}}$ is a closed interval. 
\end{proof}

In the higher genus case, 
we obtain the following corollary.

\begin{corollary}\label{cor:0154}
Let  $f$ be a $R$-closed homeomorphism on 
a closed surface with genus more than one. 
Then 
each non-periodic minimal set of $f$ 
has infinitely many connected components.  
\end{corollary}

\begin{proof}
Suppose that 
there is a non-periodic minimal set $\mathcal{M}$ of $f$ with 
at most finitely many connected components.  
Let $k$ be the number of connected components of $\mathcal{M}$. 
Then each connected component $\mathcal{M}'$ of $\mathcal{M}$ is 
a minimal set of $f^k$.  
By Lemma \ref{lem:0125}, 
we obtain that 
$\mathcal{M}'$ is a circloid. 
By Corollary \ref{cor:02}, 
we have 
$M$ is $\mathbb{S}^2$ or $\mathbb{T}^2$. 
This contradicts to the hypothesis. 
\end{proof}

\section{Main results and their proofs}

We say that 
a minimal set $\mathcal{M}$ on a surface homeomorphism 
$f : S \to S$ 
is an extension of 
a Cantor set (resp. a periodic orbit) 
if 
there are 
a surface homeomorphism $\widetilde{f}: S \to S$ 
and 
a surjective continuous map
$p: S^2 \to S^2$ which is homotopic to the identity 
such  that 
$p \circ f = \widetilde{f} \circ p$  
and $p(\mathcal{M})$ 
is a Cantor set (resp. a periodic orbit) 
which is a minimal set of $\widetilde{f}$. 
Now we state main results. 

\begin{theorem}\label{cor:0156} 
Let $M$ be a connected orientable closed surface with genus more than one. 
Each minimal set of  an $R$-closed homeomorphism on $M$   
is 
either a periodic orbit 
or 
an extension of a Cantor set. 
\end{theorem}

\begin{proof} 
Let $\mathcal{M}$ be a minimal set. 
By Lemma \ref{lem:01555}, 
$M$ is not a finite disjoint union of circloids. 
By Theorem \cite{PX}, 
we have that 
$\mathcal{M}$ is 
is an extension of either a periodic orbit or a Cantor set. 
We may assume that 
$\mathcal{M}$ is 
is an extension of a periodic orbit. 
By the proof of Addendum 3.17 \cite{JKP} 
and Proposition 5.1 \cite{PX}, 
we obtain that 
$\mathcal{M}$ has at most finitely many connected components. 
By Corollary \ref{cor:0154},  
this minimal set $\mathcal{M}$ is a periodic orbit. 
\end{proof}

In the toral case, 
we obtain the following statement. 

\begin{theorem}\label{cor:0155} 
Each $R$-closed toral homeomorphism $f$ 
satisfies one of the following: 
\\
1. 
$f$ is minimal. 
\\
2. 
$f$ is periodic. 
\\
3. 
Each minimal set is 
finite disjoint union of essential circloids.  
\\
4.  
There is a minimal set which is  
an extension of a Cantor set. 
\end{theorem}

\begin{proof} 
Suppose that 
$f$ is neither minimal nor periodic 
and 
there are no minimal sets
which are extensions of Cantor sets. 
Since $f$ is not periodic, 
by Theorem 4 \cite{JKP}, 
there is a minimal set 
$\mathcal{M}$ which is a finite disjoint union of circloids. 
Let $k$ be the number of the connected components of $\mathcal{M}$. 
By Theorem 1.1 \cite{Y2}, 
the iteration $f^k$ is also $R$-closed. 
Applying Lemma \ref{lem:01555} to $f^k$, 
we have that  
each minimal set of $f$ is 
a finite disjoint union of essential circloids. 
\end{proof} 

Recall that $f$ is aperiodic if 
$f$ has no periodic orbits. 
By Theorem D \cite{J2}, 
we obtain the following corollary. 

\begin{corollary}\label{cor:015}
Each orbits closure of 
a non-minimal aperiodic $R$-closed 
toral homeomorphism isotopic to identity 
is a circloid. 
\end{corollary}

\section{Applications to codimension two foliations}

In \cite{Y}, 
it show that 
a foliated space on a compact metrizable space 
which is minimal or ``compact and without infinite holonomy'', 
is $R$-closed.
Since each compact codimension two foliation on a compact manifold
has finite holonomy \cite{Ep} \cite{V}, 
we have that 
the set of minimal or compact codimension two foliations is 
contained in 
the set of codimension two $R$-closed foliations.   
The following examples are codimension two $R$-closed foliations 
which are neither minimal nor compact. 
%
%
%
%
Considering an axisymmetric embedding of $\mathbb{T}^2$ (resp. $\mathbb{S}^2$) into $\R^3$, 
any irrational rotation on it around the axis is 
a non-periodic $R$-closed homeomorphism. 
Taking a suspension on $\mathbb{T}^2$ (resp. $\mathbb{S}^2$), 
we obtain the following statement by Theorem \ref{cor:0155} 
(resp. Corollary \ref{lem:0155}). 

\begin{corollary}\label{cor:011} 
Each suspension 
of 
a  $R$-closed homeomorphism on $\mathbb{T}^2$ 
or $\mathbb{S}^2$  
which is neither minimal nor periodic 
induces a codimension two $R$-closed foliation which is neither minimal nor compact.  
Moreover 
there are such homeomorphisms 
on $\mathbb{T}^2$ 
and  $\mathbb{S}^2$. 
\end{corollary}


\end{document}